\documentclass[12pt]{amsart}
\usepackage{amscd,amsmath,amsthm,amssymb,graphics}
\usepackage{amscd,amsmath,amsthm,amssymb,graphics}
\usepackage{pstcol,pst-plot,pst-3d}
\usepackage{lmodern,pst-node}
\usepackage{multicol}
\usepackage{amssymb}
\usepackage{xcolor}
\usepackage{epic,eepic}
\usepackage{amsfonts,amssymb,amscd,amsmath,enumerate,verbatim}
\usepackage{pstcol,pst-plot,pst-3d}
\usepackage{enumerate}
\psset{unit=0.7cm,linewidth=0.8pt,arrowsize=2.5pt 4}

\newpsstyle{fatline}{linewidth=1.5pt}
\newpsstyle{fyp}{fillstyle=solid,fillcolor=verylight}
\definecolor{verylight}{gray}{0.97}
\definecolor{light}{gray}{0.9}
\definecolor{medium}{gray}{0.85}
\definecolor{dark}{gray}{0.6}

\def\frk{\frak}               

\def\mm{{\frk m}}

\def\Phi{{\frk n}}
\def\Phi{{\frk N}}
\def\MA{{\mathcal A}}

\def\xb{{\bold x}}

\def\opn#1#2{\def#1{\operatorname{#2}}} 
\opn\chara{char} \opn\length{\ell} \opn\pd{pd} \opn\rk{rk}
\opn\projdim{proj\,dim} \opn\injdim{inj\,dim} \opn\rank{rank}
\opn\depth{depth} \opn\grade{grade} \opn\height{height}
\opn\embdim{emb\,dim} \opn\codim{codim}

\opn\Tr{Tr} \opn\bigrank{big\,rank}
\opn\superheight{superheight}\opn\lcm{lcm}
\opn\trdeg{tr\,deg}
\opn\reg{reg} \opn\lreg{lreg} \opn\ini{in} \opn\lpd{lpd}
\opn\size{size}\opn\bigsize{bigsize}
\opn\cosize{cosize}\opn\bigcosize{bigcosize}
\opn\sdepth{sdepth}\opn\sreg{sreg}
\opn\link{link}\opn\fdepth{fdepth}
\opn\rank{rank}
\opn\Deg{Deg}
\opn\msupp{msupp}
\opn\div{div} \opn\Div{Div} \opn\cl{cl} \opn\Cl{Cl}
\let\epsilon\varepsilon
\let\phi=\varphi
\let\kappa=\varkappa
\opn\Spec{Spec} \opn\Supp{Supp} \opn\supp{supp} \opn\Sing{Sing}
\opn\Ass{Ass} \opn\Min{Min}\opn\Mon{Mon} \opn\dstab{dstab} \opn\astab{astab}
\opn\Syz{Syz}
\opn\Ann{Ann} \opn\Rad{Rad} \opn\Soc{Soc}
\opn\Im{Im} \opn\Ker{Ker} \opn\Coker{Coker} \opn\Am{Am}
\opn\Hom{Hom} \opn\Tor{Tor} \opn\Ext{Ext} \opn\End{End}
\opn\Aut{Aut} \opn\id{id}

\opn\nat{nat}
\opn\pff{pf}
\opn\Pf{Pf} \opn\GL{GL} \opn\SL{SL} \opn\mod{mod} \opn\ord{ord}
\opn\Gin{Gin} \opn\Hilb{Hilb}\opn\sort{sort}
\opn\initial{init}
\opn\ende{end}
\opn\height{height}
\opn\type{type}
\opn\set{set}
\opn\aff{aff} \opn\con{conv} \opn\relint{relint} \opn\st{st}
\opn\lk{lk} \opn\cn{cn} \opn\core{core} \opn\vol{vol}
\opn\link{link} \opn\star{star}\opn\lex{lex}
\opn\gr{gr}

\def\pot#1#2{#1[\kern-0.28ex[#2]\kern-0.28ex]}

\opn\dirlim{\underrightarrow{\lim}}
\opn\inivlim{\underleftarrow{\lim}}
\let\union=\cup
\let\sect=\cap

\let\Union=\bigcup

\let\Dirsum=\bigoplus

\let\to=\rightarrow
\let\To=\longrightarrow
\def\Implies{\ifmmode\Longrightarrow \else
        \unskip${}\Longrightarrow{}$\ignorespaces\fi}
\def\implies{\ifmmode\Rightarrow \else
        \unskip${}\Rightarrow{}$\ignorespaces\fi}
\def\iff{\ifmmode\Longleftrightarrow \else
        \unskip${}\Longleftrightarrow{}$\ignorespaces\fi}

\let\:=\colon
 \theoremstyle{plain}
\newtheorem{Theorem}{Theorem}[section]
 \newtheorem{Lemma}[Theorem]{Lemma}
 \newtheorem{Corollary}[Theorem]{Corollary}
 \newtheorem{Proposition}[Theorem]{Proposition}

 \theoremstyle{definition}
 \newtheorem{Definition}[Theorem]{Definition}

 \newtheorem{Examples}[Theorem]{Examples}
 
\let\epsilon\varepsilon
\let\kappa=\varkappa
\opn\dis{dis}
\def\pnt{{\raise0.5mm\hbox{\large\bf.}}}

\opn\Lex{Lex}

\begin{document}
\title{Restricted classes of Veronese type ideals and algebras}
\author {Rodica Dinu, J\"urgen Herzog,  Ayesha Asloob Qureshi}

\address{Rodica Dinu, Faculty of Mathematics and Computer Science, University of Bucharest, Str. Academiei 14, 010014 Bucharest, Romania}
\email{rdinu@fmi.unibuc.ro}

\address{J\"urgen Herzog, Fachbereich Mathematik, Universit\"at Duisburg-Essen, Fakult\"at f\"ur Mathematik, 45117 Essen, Germany} \email{juergen.herzog@uni-essen.de}

\address{Ayesha Asloob Qureshi, Sabanc\i \; University, Faculty of Engineering and Natural Sciences, Orta Mahalle, Tuzla 34956, Istanbul, Turkey}\email{aqureshi@sabanciuniv.edu}


\begin{abstract}
We study ideals  which are  generated by   monomials of degree $d$ in the polynomial ring in $n$ variables and  which satisfy certain numerical side conditions regarding their exponents. Typical examples of such ideals are the ideals of Veronese type, squarefree Veronese ideals or $t$-spread Veronese ideals.   In this paper we focus on $c$-bounded  $t$-spread Veronese ideals and on Veronese ideals of bounded support.  Their powers as well as their fiber cone is also  be considered.
\end{abstract}

\thanks{This paper was written during the visit of the first author at the Universit\" at Duisburg-Essen.}
\subjclass[2010]{05E40, 13C14, 13D02}
\keywords{c-bounded t-spread Veronese ideals, analytic spread, t-spread Veronese ideals of bounded type}

\maketitle
\section*{Introduction}

The study of Veronese ideals and Veronese algebras, and variations of them,  has a long history. The motivation to study these algebraic objects comes from combinatorics and algebraic geometry. In this paper we aim at a unified treatment,  and then focus on two particular classes  of Veronese type ideals and algebras.

Let $K$ be a field and $S=K[x_1, \dots, x_n]$ be the polynomial ring in $n$ indeterminates over $K$.
We fix an integer $d$ and a sequence $\mathbf{a}=(a_1,\dots, a_n)$ of integers with $1\leq a_1\leq \dots \leq a_n \leq d$ and $d<\sum_{i=1}^{d}a_i$. The monomial ideal in $S$ generated by all the monomials of the form $x_1^{c_1}\cdots x_n^{c_n}$ with $c_1+\dots +c_n=d$ and with $c_i\leq a_i$ for all $1\leq i\leq n$ is called an {\it ideal of Veronese type}. We denote it by $I_{n,d, (a_1,\dots, a_n)}$. Properties of these ideals have been  studied for example in \cite{HHV}, \cite{HRV} and  \cite{assocprim}. The $K$-subalgebra of $S$ generated by the monomials $u\in G(I_{n,d,(a_1,\dots,a_n)})$ is called an  {\it algebra of Veronese type}. Here, for a monomial  ideal $I$, we denote by $G(I)$ the set of minimal monomial  generators of  $I$. De Negri and Hibi, in \cite{deNegri}, characterized those algebra of Veronese type which are Gorenstein.

While ideals of Veronese type are defined by bounding the exponents of the generators, the next family is defined by spreading the variables. In paper \cite{ehq}, the concept of a $t$-spread monomial was introduced. A monomial $x_{i_1}x_{i_2}\cdots x_{i_d}$ with $i_1\leq i_2\leq \dots \leq i_d$ is called {\it $t$-spread} if $i_j -i_{j-1}\geq t$ for $2\leq j\leq n$. We fix integers $d$ and $t$. The monomial ideal in $S$ generated by all $t$-spread monomials of degree $d$ is called the {\it t-spread Veronese ideal of degree $d$}. We denote it by $I_{n,d,t}$. For $t=1$ one obtains the squarefree Veronese ideals, which may also be viewed as the edge ideal of a hypersimplex. Properties of these ideals were first studied  in \cite{Sturmfels},  and then further in \cite{ehq}, \cite{bahareh}, \cite{JZ} and  \cite{JKMR}. The $K$-subalgebra of $S$ generated by the monomials $v\in G(I_{n,d,t})$ is called the {\it $t$-spread Veronese algebra}. In  \cite{dinu} the Gorenstein property for the $t$-spread Veronese algebra was analyzed.

In this paper we present a  unified concept to deal with all the above cases  and to study new such classes.

In Section~\ref{bc} this unified point of view is introduced by identifying  multisets with  monomials. For a given multiset $A=\{i_1\leq i_2\leq \cdots \leq i_d\}$,  the corresponding monomial is $x_{i_1}x_{i_2}\cdots x_{i_d}$.
The multiset $A$ is called {\em $t$-spread}, if $i_{k+1}-i_k \geq t$ for all $k$. In that case  a subset $B \subset A$ is called a {\em block} of size $r$ if $B=\{i_k\leq i_{k+1}\leq \ldots \leq i_{k+r-1}\}$ with $i_{l+1}-i_l=t$ for all $l$. Any multiset has a unique block decomposition into maximal blocks. With this notion at hand one defines different restricted  classes of  Veronese ideals. We denote by  $\MA_{n,d,t}$, the set of all $t$-spread multisets $A \in \MA_{n,d}$, where $\MA_{n,d}$ denotes multisets in $[n]$ with $d$ elements.  For  given positive integers  $c$ and $k$, we consider the following two classes of  multisets:

1.  $\MA_{c,(n,d,t)}$ is the set of all multisets $A \in \MA_{n,d,t}$ such that $|B|\leq c$ for each block $B \subset A$, and we set $I_{c,(n,d,t)}= (\xb_A \:\; A \in \MA_{c,(n,d,t)})$.

2. $\MA_{(n,d,t),k}$ is the set of all multisets $A \in \MA_{n,d,t}$ which have at most $k$ blocks in their unique block decomposition, and we set $I_{(n,d,t),k}= (\xb_A \:\; A \in \MA_{(n,d,t),k}).$

The first family of multisets  gives us the $c$-bounded $t$-spread Veronese ideals which we study in Section~\ref{ctVero}. It is shown that these ideals have linear quotients (Theorem~\ref{asbefore}). In Corollary~\ref{macgor}, we characterize in terms of $c,n,d$ and $t$ when $I_{c,(n,d,t)}$ is  Cohen--Macaulay or  Gorenstein.

Section~\ref{ctVerofiber} is devoted to the study of the powers of the ideal $I_{c,(n,d,t)}$  and its fiber cone. The basic tool for this task is provided  in Theorem~\ref{sortable}, where it is shown that the minimal monomial set of generators of $I_{c,(n,d,t)}$ is a sortable set. This has several important consequences.  First of all it follows in Corollary~\ref{koszul} that the fiber cone of $I_{c,(n,d,t)}$ is Koszul and a Cohen-Macaulay  normal domain. Together with the fact that $I_{c,(n,d,t)}$  satisfies the so-called {\it $l$-exchange property} with respect to the sorting order, as shown in Theorem~\ref{thmx}, we conclude in Corollary~\ref{ReesCM} that the Rees algebra $\mathcal{R}(I_{c,(n,d,t)})$ is a normal Cohen-Macaulay domain. This has the  consequence that $I_{c,(n,d,t)}$ satisfies the strong persistence property and that all powers of $I_{c,(n,d,t)}$ have linear resolution, see Corollary~\ref{strongpersistence}. These results generalize the corresponding statements for $t$-spread  Veronese ideals in \cite{ehq} and follow the similar line of arguments, though the $c$-bound condition, requires substantial extra efforts in the proofs.

In order to compute the analytic spread  of $I_{c,(n,d,t)}$,  which is the Krull dimension of its fiber cone, we present  in Lemma~\ref{better}  a result, which is of quite general interest and generalizes \cite[Lemma 4.2]{JQ}. It is shown that for a monomial ideal $I$ with linear relations which is  generated in a single degree and whose linear relation graph $\Gamma$ has $r$ vertices and $s$ connected components,  the analytic spread of $I$ is given by the formula  $\ell(I)=r-s+1$. Based on this lemma we succeed in Theorem~\ref{headache} to give a complete answer regarding the analytic spread of $I_{c,(n,d,t)}$. For its  proof one introduces  a partial order on the monomial generators of $I_{c,(n,d,t)}$, by saying that $u$ covers $v$ if there exist $i<j$ such that $x_j$ divides $u$ and $v=x_i(u/x_j)$. In Theorem~\ref{maxelem} is shown that there exists a unique maximal element  and a unique minimal with respect to this partial order. The exponents of these elements as well as the integers $c,n,d$ and $t$ determine in  a complicated  but explicit way the analytic spread of $I_{c,(n,d,t)}$. Several  examples demonstrate the result.

In the last section of the paper we study $t$-spread Veronese ideals of bounded block type, namely the ideals of the form    $I_{(n,d,t),k}$.  This is the $t$-spread  monomial ideal in $n$ variables generated in degree $d$ with at most $k$ blocks. In particular,  $I_{(n,d,0),k}$ is the ideal in $n$ variables generated in degree $d$  whose generators have support of  cardinality at most  $k$.

For a suitable integer $m$, the ideal $I_{(n,d,t),k}$ is obtained from $I_{(m,d,0),k}$  by an iterated application of the Kalai shifting  operator. In this paper we concentrate on the study of the ideals $I_{(n,d,0),k}$. It can be easily seen
$I_{(n,d,0),k} = \sum _{i_1 \leq i_2 \leq \ldots \leq i_k} (x_{i_1}, \ldots, x_{i_k})^d$.

These ideals are of height $n$, so that the regularity is determined by their socle degree. The regularity of $I_{(n,d,0),k}$ is given in Theorem~\ref{regblock}. At present we do not know the regularity of $I_{(n,d,t),k}$ when $t>0$, even it is obtained by shifting from an ideal of the type $I_{(m,d,0),k}$. Indeed since $I_{(m,d,0),k}$ is not a stable ideal, there is no obvious relationship between its regularity  and that of its shifted ideals. In Proposition~\ref{k=2} we show that $I_{(n,d,0),k}^{n-1}$ is a power of the maximal ideal, so that this power of $I_{(n,d,0),k}^{n-1}$ has linear resolution. In general much smaller powers of   $I_{(n,d,0),k}$ may have linear resolution.  Knowing this smallest power $j$  such that $I_{(n,d,0),k}^j$ has linear resolution gives us the regularity of the fiber cone
of $I_{(n,d,0),k}$, see Corollary~\ref{fibercone}. The examples that we could check with Singular\cite{Si} indicate that the fiber cone of  $I_{(n,d,0),k}$ has quadratic relations. Among them is the well-known pinched Veronese which is known to be even Koszul, see \cite{caviglia}.

\section{Basic concepts}\label{bc}
Let $\MA_{n,d}$ be the set of all multisets $A \subset [n]$ with $|A|=d$. A multiset
\[
A=\{i_1\leq i_2\leq \ldots \leq i_d\}\subset [n]
\]
 is called $t$-spread if $i_{k+1}-i_k \geq t$ for all $k$.

 We denote by $\MA_{n,d,t}$, the set of all $t$-spread $A \in \MA_{n,d}$. Note that $\MA_{n,d,0}=\MA_{n,d}$.
 Let $A \in \MA_{n,d,t}$. A subset $B \subset A$ is called a {\em block} of size $r$ if $B=\{i_k\leq i_{k+1}\leq \ldots \leq i_{k+r-1}\}$ with $i_{l+1}-i_l=t$ for all $l$. The block $B$ is called {\em maximal} if it is not properly contained in any other block of $A$. Note that every $A \in \MA_{n,d,t}$ has unique {\em block decomposition}. In other words, $A=B_1 \sqcup B_2\sqcup \ldots \sqcup B_k$, where each $B_i$ is a maximal block. The number of blocks in $A$ is called {\em the block type of $A$}. For example, $\{ 1,3,5,8,10,13,16\} \subset \MA_{16, 7, 2}$ has block decomposition
$ B_1=\{1,3,5\}$, $B_2= \{8,10\}$, $B_3 =\{13\}$ and $B_4=\{16\}.$

In this paper, we consider, the following restricted multisets of $\MA_{n,d,t}$. Given positive integers $c$ and $k$, we define:
\begin{enumerate}
\item $\MA_{c,(n,d,t)}$ is the set of all multisets $A \in \MA_{n,d,t}$ such that $|B|\leq c$ for each block $B \subset A$.
\item $\MA_{(n,d,t),k}$ is the set of all multisets $A \in \MA_{n,d,t}$ which have at most $k$ blocks in their unique block decomposition.
\end{enumerate}

Let $K$ be a field and $S=K[x_1, \ldots, x_n]$ be the polynomial ring in $n$ indeterminates. For a multiset  $A=\{i_1\leq i_2\leq \ldots \leq i_r\}$ on $ [n]$, we define the monomial $\xb_{A}=\prod_{j=1}^rx_{i_j}$.

Corresponding to the restricted multisets introduced in (1) and (2), we define the following monomial ideals
\begin{center}
$I_{c,(n,d,t)}= (\xb_A \:\; A \in \MA_{c,(n,d,t)}) \quad \text{and} \quad I_{(n,d,t),k}= (\xb_A \:\; A \in \MA_{(n,d,t),k}).$
\end{center}
\begin{Examples}
(i) The ideal $I_{c,(n,d,0)}$ is of Veronese-type $I_{n,d;(c,\ldots, c)}$. In other words, the ideal $I_{c,(n,d,0)}$ is generated by all monomials in degree $d$ whose exponents are bounded by $c$.

(ii) The ideal $I_{1, (n,d,0)}$ is the squarefree Veronese ideal.

(iii) The ideal $I_{(n,d,0),k}$ is the ideal generated by all $u \in \Mon(S)$ of degree $d$ such that $|\supp(u)| \leq k$.
\end{Examples}

Here, as usual,  the support of a monomial $u$ is the set $\{i\:\, x_i|u\}$. We denote this  set by $\supp(u)$.

In what follows we  also need the concept of the multi-support of a monomial. Let $u=x_{i_1}x_{i_2}\dots x_{i_d}$ with $i_1\leq i_2\leq \dots \leq i_d$. We call the multiset $\{i_1, \dots, i_d\}$ the multi-support of $u$ and denote it by $\msupp(u)$.

\section{$c$-bounded $t$-spread Veronese ideals}\label{ctVero}
In this section we study properties of $c$-bounded $t$-spread Veronese ideals. Recall that $I_{c,(n,d,t)}$ denotes  the $t$-spread monomial ideal in $n$ variables generated in degree $d$ whose blocks are $c$-bounded.

\begin{Theorem}\label{asbefore}
The ideal  $I_{c,(n,d,t)}$ has linear quotients with respect to the lexicographic order.
\end{Theorem}

\begin{proof} For simplicity, we set $I=I_{c,(n,d,t)}$.
Let $G(I)=\{u_1, \dots, u_m\}$ ordered with respect to the lexicographic order. Let $r\leq m$ and $J=(u_1, \dots, u_{r-1})$. We want to show that $J: u_r$ is generated by variables. In this context, it is enough to prove that for all $1\leq k\leq r-1$ there exists $x_i \in J:u_r$ such that $x_i$ divides $u_k/\gcd(u_k, u_r)$. Let $u_k=x_{i_1}x_{i_2}\dots x_{i_d}$ with $i_1\leq i_2\leq \ldots \leq i_d$ and $u_r= x_{j_1}x_{j_2}\dots x_{j_d}$ with $j_1\leq j_2\leq \ldots \leq  j_d$. Since $u_k >_{lex} u_r$ there exists $q$ with $1\leq q\leq t$ such that $i_1=j_1$, $\dots$, $i_{q-1}=j_{q-1}$ and $i_q<j_q$. Let $v=x_{i_q}(u_r/x_{j_q})$.
Then, $v=x_{j_1}x_{j_2}\dots x_{j_{q-1}}x_{i_{q}}x_{j_{q+1}}\dots x_{j_d}=x_{i_1}x_{i_2}\dots x_{i_{q-1}}x_{i_{q}}x_{j_{q+1}}\dots x_{j_d}$. Since $j_{q+1}-i_q> j_{q+1}-j_q\geq t$, we see that $v$ is $t$-spread and any maximal block belongs to $\{j_1, \dots , j_{q-1},i_q\}= \{i_1, \dots , i_{q-1},i_q\}$ or $\{j_{q+1}, \dots, j_d\}$. Hence, since $u_k$ and $u_r$ are $c$-bounded, it follows that $v$ is also $c$-bounded. Thus, $v\in G(J)$ because $v\in G(I)$ and $v>_{\lex} u_r$. Therefore, $x_{i_q}\in J: u_r$ and $x_{i_q}$  divides $u_k/\gcd(u_k, u_r)$, as required.
 \end{proof}

Let $I=I_{c,(n,d,t)}$   and $u \in I$,    and   let $B_1\sqcup\cdots \sqcup B_r$ be the block decomposition of $\msupp(u)$.  For $j=1,\ldots,r-1$,  the  {\em $j$th gap  interval} of $u$ is the interval
$[ \max(B_j)+t,\min(B_{j+1})-1]$  if $|B_j|<c$ and is the interval $[ \max(B_j)+t+1,\min(B_{j+1})-1]$ if $|B_j|=c$. Furthermore, we call $[1,\min(B_1)-1]$ the $0$th gap interval.

 \medskip
By using this terminology we have

\begin{Lemma}
\label{gap}
Let $u\in G(I)$. Then $i\in \set(u)$ if and only if $i$ belongs to some gap interval of $u$.
\end{Lemma}

\begin{proof}
Let  $i $ be an integer that belongs to the $j$th gap interval of $u$,  and  let  $k=\min(B_{j+1})$ and    $v=x_i u / x_{k}$. It  follows from the definition of the gap intervals that $v\in I$.  Since  $v >_{\lex} u$, it follows that $i\in \set(u)$.

Let $i\in \set(u)$. Then there exists $v=x_iu/x_{k}\in I$ for some $i<k$. In particular, $i < k \leq \max(u)=\max(B_r)$. We have the following:
\begin{enumerate}
\item $i \notin [\min(B_r), \max(B_{r})]$, otherwise $v$ is not $t$-spread if $t\geq 1$, and $i=k$ if $t=0$,  which in both cases is a contradiction.

\item For any $1 \leq j \leq r-1$ with $|B_j|< c$, we have $i \notin[\min(B_j), \max(B_j)+t-1]$. Otherwise, if $t \geq 1$ then $v$ is not $t$-spread, and  if $t=0$, then $[\min(B_j), \max(B_j)-1]=\emptyset$ because $\min(B_j)=\max(B_j)$, which in both cases is a contradiction.
\item  For any $1 \leq j \leq r-1$ with $|B_j|= c$, we have $i \notin [\min(B_j), \max(B_j)+t]$. Otherwise, $v$ is not $c$-bounded if $i= \max(B_j)+t$. Furthermore, if $t\geq 1$ then $v$ is not $t$-spread for $\min(B_j) \leq t < \max(B_j)+t$, and if $t=0$, then again $v$ is not $c$-bounded, which again in all cases leads to a contradiction.

\end{enumerate}

The only possibility that remains is that $i$ belongs to a $j$th gap interval for some $0\leq j \leq r-1$.
\end{proof}

\begin{Examples}
(i) Let $u=x_1x_2x_3x_6x_{10} \in I_{3, (10, 5, 1)}$. Then $\set(u)=\{5,7,8,9\}$.\\
(ii) Let $u=x_1^3x_3^2x_5 \in I_{3, (5,6,0)}$. Then $\set(u)=\{2,3,4\}$.
\end{Examples}

 We define the {\em Kalai stretching operator} $\sigma$ as follows: for $A=\{i_1\leq i_2\leq \dots\leq i_d\}\in [n]$, $A^{\sigma}$ is the multiset $\{i_1, i_2+1,\dots, i_d+d-1\}\subseteq [n+d-1]$. Inductively, for any $t\geq 1$, we define $A^{\sigma^{t}}=(A^{\sigma^{t-1}})^{\sigma}$. For $u=x_{A}$, we set $u^{\sigma^t}=x_{A^{\sigma^t}}$. We define the inverse map of $\sigma$ by $\tau$ as follows: for $A=\{i_1< i_2< \dots< i_d\}\in [n]$, $A^{\tau}$ is the set $\{i_1, i_2-1,\dots, i_d-(d-1)\}\subseteq [n-(d-1)]$.

 \begin{Proposition}\label{shift} The stretching operator $\sigma$  has the following properties:
 \begin{enumerate}[{\em(i)}]
\item $ v\in G(I_{c, (n+d-1, d, t+1)})$ if and only if there exists $u\in G(I_{c,(n,d,t)})$ such that $v= u^{\sigma}$;
\item Let $u\in G(I_{c,(n,d,t)})$, $u=x_{A}$ and  $A=\{i_1\leq \dots \leq i_d\}$. Assume $\set(u)=\{a_1, \dots, a_r\}$. Then $\set(u^{\sigma})=\{b_1, \dots, b_r\}$, where $b_i=a_i+\max\{j: i_j \leq a_i\}$ for all $i$ and $\max \emptyset=0$.
  \end{enumerate}
 \end{Proposition}

 \begin{proof}
 (i) If  $A$ is $t$-spread, it follows that $A^{\sigma}$ is $t+1$-spread. Moreover, if $B_1\sqcup \dots \sqcup B_r$, is the block decomposition of $A$, then the block decomposition $A^{\sigma}$ is  $C_1\sqcup\cdots \sqcup C_r$ such that if $B_j=\{i_k \leq i_{k+1} \leq \ldots \leq {i_l}\}$ then $C_j= \{i_k+k-1 < i_{k+1}+k < \ldots < i_l+l-1\}$. This shows that  $v\in I_{c, (n+d-1, d, t+1)}$ if $u\in I_{c,(n,d,t)}$. Similarly one shows that if $ v\in I_{c, (n+d-1, d, t+1)} $ and  $u=v^{\tau}$, then $u\in I_{c, (n, d, t)}$. Since $u^{\sigma}=v$, this completes the proof.

 (ii)   From Lemma~\ref{gap}, we know $\set(u)$ and $\set(u^{\sigma})$ are equal to the union of all gap intervals of $A$ and $A^{\sigma}$, respectively. From the construction of $A$ and $A^{\sigma}$ in part (i), we see that $A$ and $A^{\sigma}$ has same number of blocks and hence the same number of gap intervals. Note that the $0$th gap interval is same for $A$ and $A^{\sigma}$ because $\min(B_1)=\min(C_1)$. Therefore, $a_i$ belongs to the $0$th gap interval of $u$ if and only if $a_i=b_i$ belongs to the $0$th gap interval of $u^{\sigma}$.

For $j>0$, by following the definition of gap intervals, the $j$th gap interval is given as follows :
\begin{equation}\label{1}
a_i \in \{i_l+p, i_l+p+1, \ldots, i_{l+1}-1\},
\end{equation}
 where $i_l =\max(B_j)$ and $p=t$ if $|B_j|<c$ and $p=t+1$ if $|B_j|=c$. Note that $\max\{j: i_j \leq a_i\}=l$. Furthermore, the $j$th gap interval of $u^{\sigma}$ is
 \[
 \{i_l+ (l-1)+q, i_l+(l-1)+q+1, \ldots, i_{l+1}+l-1\},
 \]
 where $i_l+l-1= \max(C_j)$ and $q=t+1$ if $|C_j|<c$ and $q=t+2$ if $|C_j|=c$. Since $|B_j|=|C_j|$, we have $q=p+1$. So we may write the $j$th gap interval of $u^{\sigma}$ as
\begin{equation}\label{2}
 \{i_l+p+l, i_l+p+1+l, \ldots, i_{l+1}-1+l\}.
\end{equation}
Therefore, for any $j>0$, $a_i$ is in $j$th gap interval of $u$  if and only if $b_i=a_i+l$ is in $j$th gap interval of $u^{\sigma}$.

 \end{proof}

\begin{Corollary}\label{burned}
$\beta_{i}(I_{c,(n,d,t)})= \beta_{i}(I_{c,(n-(d-1)t,d,0)})$ for all $i$.
\end{Corollary}

\begin{proof}
We use the following general fact (see \cite[Lemma 1.5]{JT}): Let $I$ be a monomial ideal generated in degree $d$ with linear quotients. Then
\begin{equation}\label{eq}
\beta_{i}(I)=|\{\alpha \subset \set(u): u\in G(I) \text{ and } |\alpha|=i\}|.
\end{equation}

Now, let $G(I_{c,(n-(d-1)t,d,0)})=\{u_1,\dots, u_r\}$, then $G(I_{c,(n,d,t)})=\{u_1^{\sigma^t}, \dots, u_r^{\sigma^t}\}$, by Proposition \ref{shift}(i), and by Proposition \ref{shift}(ii) we have $|\set(u_j)|=|\set(u_j^{\sigma^t})|$ for all $j$. Thus, the desired conclusion follows from \eqref{eq}.
\end{proof}

\begin{Lemma}\label{height}
Let $I\subset S$ be  a graded ideal in $S=K[x_1, \dots, x_n]$ and $J \subset T$ a graded ideal in $T=K[x_1,\dots, x_{n^{\prime}}]$ such that $\beta_{i,j}(I)=\beta_{i,j}(J)$ for all $i$ and $j$. Then
\[
\height(I)=\height(J).
\]
\end{Lemma}

\begin{proof}
Denote by $d= \dim(S/I)$ and by $d^{\prime}=\dim(T/J)$. Since $I$ and $J$ have the same graded Betti numbers, we have that $\Hilb(S/I)=\frac{P(t)}{(1-t)^n}=\frac{Q(t)}{(1-t)^{d}}$, $Q(1)\neq 0$, and $\Hilb(T/J)=\frac{P(t)}{(1-t)^{n^{\prime}}}=\frac{Q^{\prime}(t)}{(1-t)^{d^{\prime}}}$, $Q^{\prime}(1)\neq 0$. Thus,
 $P(t)=Q(t)(1-t)^{n-d}=Q^{\prime}(t)(1-t)^{n^{\prime}-d^{\prime}}$. Hence, $n-d= n^{\prime}-d^{\prime}$, which leads to the desired conclusion.
\end{proof}

\begin{Proposition}
\label{height}
{\em(i)} $\height(I_{c,(n,d,t)})= \height(I_{c,(n-(d-1)t,d,0)}).$

{\em (ii)} $\height(I_{c,(n,d,t)})=n-(k+t(d-1))$, where $k=\lfloor \frac{d-1}{c}\rfloor$ and $r=d-kc$.

{\em (iii)} $n-(k+t(d-1))=\max\{\min(u) \:\; u \in G(I_{c,(n,d,t)})\}$.
\end{Proposition}

\begin{proof}
(i) By Corollary \ref{burned} the ideals $I_{c,(n,d,t)}$ and  $I_{c,(n-(d-1)t,d,0)}$   have the same  Betti numbers and  by Theorem~\ref{asbefore} they both have $d$-linear resolution. Therefore,  their the graded Betti numbers are the same. Thus,  by applying Lemma \ref{height}, we conclude that the two ideals have the same height. Part (ii) follows from part (i) and \cite[Proposition 3.1]{assocprim}. For part (iii), it is easy to see that the smallest monomial generator of $I_{c,(n,d,t)}$ with respect to the lexicographic order is the monomial $$x_{n-k((t+1)+(c-1)t)-rt} \cdots x_{n-(c-1)t-(t+1)}x_{n-(c-1)t}\cdots x_{n-t}x_n.$$ Thus we obtain the desired conclusion.
\end{proof}

\begin{Corollary}
\label{macgor}
{\em (i)} $I_{c,(n,d,t)}$ is Cohen-Macaulay if and only if $n\leq t(d-1)+\lfloor\frac{d-1}{c}\rfloor+1$, or $d \leq c$, or $c=1$.

{\em (ii)} $I_{c,(n,d,t)}$ is Gorenstein if  and only if $n\leq t(d-1)+\lfloor\frac{d-1}{c}\rfloor+1$ or  $d=1$.
\end{Corollary}

\begin{proof}
(i) By Corollary \ref{burned}, the ideals $I_{c,(n,d,t)}$ and $I_{c,(n-(d-1)t,d,0)}$ have the same graded Betti numbers and hence the same  projective dimension. Therefore,  by Proposition \ref{height}, they have the same height. It follows that $I_{c,(n,d,t)}$ is Cohen-Macaulay if and only if $I=I_{c,(n-(d-1)t,d,0)}$ is Cohen-Macaulay. Therefore, the desired conclusion follows from \cite[Theorem 4.2]{CM}, which says that the Veronese type ideal is Cohen-Macaulay if and only if $I$ is either a principal ideal and hence $\height(I)\leq 1$, which  by Proposition \ref{height} means that  $n\leq t(d-1)+\lfloor\frac{d-1}{c}\rfloor+1$, or $I$ is a power of the  maximal ideal, which is the case if  and only if $d\leq c$, or $I$ is  squarefree Veronese, which is the case if and only if $c=1$.

(ii) Since  the resolution for $I_{c,(n,d,t)}$ is $d$-linear and since for Gorenstein rings the resolution for $S/I_{c,(n,d,t)}$ is self-dual it follows that $I_{c,(n,d,t)}$ is Gorenstein if and only if $d=1$ or $I_{c,(n,d,t)}$ is principal, i.e. $n\leq t(d-1)+\lfloor\frac{d-1}{c}\rfloor+1$.
\end{proof}

\section{On the powers and the fiber cone of $c$-bounded $t$-spread Veronese ideals}\label{ctVerofiber}

 We start this section by defining sorted sets of monomials, a notion due to Sturmfels \cite{Sturmfels}. The relations of toric rings generated by sortable sets are well understood.
Let $S_d$ be the $K$-vector space generated by all the monomials of degree $d$ in $S$ and let $v, w \in S_d$. We write $vw=x_{i_1}\dots x_{i_2}\dots x_{i_{2d}}$ with $i_1\leq i_2 \leq \dots \leq i_{2d}$. The {\it sorting} of the pair $(v,w)$ is the pair of monomials $(v^{\prime}, w^{\prime})$ with

\begin{center}
$v^{\prime}=x_{i_1}x_{i_3}\dots x_{i_{2d-1}}, w^{\prime}= x_{i_2}x_{i_4}\dots x_{i_{2d}}.$
\end{center}
The map
\begin{center}
$\sort: S_d\times S_d \rightarrow S_d\times S_d, (u,v)\mapsto (u^{\prime}, v^{\prime})$
\end{center}
is called the {\it sorting operator}. A pair $(u,v)$ is {\it sorted} if $\sort(u,v)=(u,v)$. Notice that if $(u,v)$ is sorted, then $u>_{lex}v$ and $\sort(u,v)=\sort(v,u)$. If $u_1=x_{i_1}\dots x_{i_d}, u_2=x_{j_1}\dots x_{j_d}, \dots, u_{r}=x_{l_1}\dots x_{l_d}$, then the $r$-tuple $(u_1, \dots, u_r)$ is sorted if and only if
\begin{eqnarray}
\label{sequence}
i_1\leq j_1\leq  \dots\leq  l_1\leq i_2\leq j_2\leq \dots \leq l_2\leq \dots \leq i_d\leq j_d\leq \dots \leq l_d,
\end{eqnarray}
which means that $(u_i, u_j)$ is sorted, for all $i>j$.
For any $r$-tuple of monomials $(u_1, \dots, u_r)$, $u_i \in S_d$, for all $i$, there exists a unique sorted $r$-tuple $(v_1,\dots, v_r)$ such that $u_1\cdots u_r=v_1\cdots v_r$, see \cite[Theorem 6.12]{EH}. We say that the product $u_1\cdots u_r$ is sorted if $(u_1, \dots, u_r)=(v_1, \dots, v_r)$.


\begin{Theorem}
\label{sortable}
The set $G(I_{c,(n,d,t)})$ is sortable.
\end{Theorem}
\begin{proof}
 Let $u,v\in G(I_{c,(n,d,t)})$. Then $u=\bf{x}_{U}$ and $v=\bf{x}_{V}$ for some $U,V \in\MA_{c,(n,d,t)}$. Let $uv=x_{i_1}x_{i_2}\cdots x_{i_{2d}}$ such that $i_1\leq i_2 \leq \ldots \leq i_{2d}$. Moreover, let $v^\prime=x_{i_1}x_{i_3}\cdots x_{i_{2d-1}}$ and $u^\prime =x_{i_{2}}x_{i_4}\cdots x_{i_{2d}}$. We need to show that $u', v' \in G(I_{c,(n,d,t)})$.

If $t=0$, then note that $\deg_{x_i}(u')$ and $\deg_{x_i}(v')$ is at most $\left\lceil \frac{\deg_{x_i} (u)+ \deg_{x_i} (v)}{2}\right\rceil  $. Since $\deg_{x_i} (u), \deg_{x_i} (u) \leq c$, it shows that  $\deg_{x_i}(u')$ and $\deg_{x_i}(v')$ is also bounded by c. This shows that $u', v' \in G(I_{c,(n,d,0)})$.

 Let $t \geq 1$. It follows from \cite[Proposition 3.1]{ehq} that $u'$ and $v'$ are $t$-spread monomials. In particular, $u'$ and $v'$ are squarefree. Now we will show that if $B \subset \{i_2< i_4< \ldots< i_{2d}\}$ is a block then $|B| \leq c$. The case when $B \subset \{i_1< i_3< \ldots< i_{2d-1}\}$ follows by a similar argument.

Let $B=\{i_{2l} < i_{2(l+1)}< \ldots <i_{2(l+k)}\}$.  for some $1\leq l < l+k \leq d$. Let $A=\{i_{2l}\leq i_{2l+1}\leq i_{2l+2}\leq i_{2l+3}\leq  \ldots \leq i_{2(l+k)-1}\leq i_{2(l+k)}\}$. We have either $i_{2l} \in U$ or $i_{2l} \notin U$. Assume that  $i_{2l} \in U$ . Then we have two possibilities, namely,  either $i_{2l+1} \in U$ or $i_{2l+1} \notin U$. If $i_{2l+1} \in U$ then $i_{2l+1}-i_{2l}=t$ and $i_{2l+1}=i_{2l+2}$ because $i_{2l+2}-i_{2l}=t$ and $i_{2l} \leq i_{2l+1}\leq i_{2l+2}$. Therefore, $i_{2l+2} \in U$.
 On the other hand, if $i_{2l+1} \notin U$ then  $i_{2l}< i_{2l+1}$ and  $i_{2l+1} \in V$. Again, in this case $i_{2l+2} \in U$ because $i_{2l+2} \notin V$. Indeed $i_{2l} < i_{2l+1} \leq i_{2l+2}$ and therefore $i_{2l+2} -i_{2l+1} < t$. We see that in both cases, if $i_{2l} \in U$ then $i_{2l+2} \in U$. Continuing in the same way, we see that $B \subset U$ and hence $|B| \leq c$.

Now assume that  $i_{2l} \notin U$. Then $i_{2l} \in V$ and by following the same argument as above, we conclude that  $B \subset V$ and $|B| \leq c$.


\end{proof}

Let $I$ be an equigenerated monomial ideal.  We consider the polynomial ring $T=K[\{t_u \:\; u\in G(I)\}]$ in $|G(I)|$ variables.  We denote by $K[I]$ the $K$-algebra which  is generated by the set of monomials $G(I)$.  The kernel of the $K$-algebra homomorphism
\begin{center}
$\varphi : T \rightarrow K[I]$, $\varphi(t_u)=u,$ for $u\in G(I)$
\end{center}
is a binomial ideal and is called the defining ideal of $K[I]$.

\begin{Corollary}
\label{koszul}
$K[I_{c,(n,d,t)}]$ is Koszul and a Cohen-Macaulay normal domain.
\end{Corollary}

\begin{proof}
We use the fact that the sorting relations form a quadratic Gr{\"o}bner basis of the defining $J$ of $K[I_{c,(n,d,t)}]$  with respect to the sorting order, see
\cite[Theorem 6.15]{EH} and \cite[Theorem 6.16]{EH}.  Then, by using a result of Fr\"oberg \cite{Froberg}, the algebra is Koszul. Since the initial ideal of $J$ with respect to the sorting order is squarefree it follow from  \cite{Sturmfels}, see also \cite[Corollary 4.26]{binomialideals},  that $K[I_{c,(n,d,t)}]$ is normal, and by \cite{Hochster} this implies  that it is Cohen-Macaulay.
\end{proof}

In order to study the powers of the ideals $I_{c,(n,d,t)}$, we have to understand the structure of the Rees algebra of such ideals. To do this, we use the so-called {\it $l$-exchange property}, see \cite{HHV} or \cite[Section 6.4]{EH}.

Let $\mathcal{R}(I)=\bigoplus_{j\geq 0}I^{j}t^j\subset S[t]$ be the Rees algebra of  an equigenerated monomial  ideal $I$. Since $I$ is equigenerated,  the fiber $\mathcal{R}(I)/\mathfrak{m}\mathcal{R}(I)$ of the Rees algebra $\mathcal{R}(I)$ is isomorphic to the toric $K$-algebra $K[I]$.

The Rees ring $\mathcal{R}(I)$ has the presentation
\begin{center}
$\psi : R=S[\{t_u \:\; u\in G(I)\}] \rightarrow \mathcal{R}(I)$,
\end{center}
defined by
\begin{center}
$x_i \mapsto x_i$, for $1\leq i\leq n$ and $t_v \mapsto vt$, for $v\in G(I)$.
\end{center}

Let $<$ be a monomial order on $T$. We call a monomial $t_{u_1}\cdots t_{u_N}\in T$ {\it standard with respect to $<$},  if it does not belong to the initial ideal of the defining ideal $J\subset T$  of $K[I]$.

\begin{Definition}\cite{HHV} \label{lexchange}
A monomial ideal $I\subset S$ satisfies the $l$-exchange property with respect to the monomial order $<$ on $T$,  if the following conditions are satisfied: let $t_{u_1}\cdots t_{u_N}, t_{v_1}\cdots t_{v_N}$ be two standard monomials in $T$  of degree $N$  with respect to $<$ such that:
\begin{enumerate}[{(i)}]

\item $\deg_{x_i}u_1\cdots u_N=\deg_{x_i}v_1\cdots v_N$, for $1\leq i \leq q-1$ with $q\leq n-1$.
\item $\deg_{x_q}u_1\cdots u_N<\deg_{x_q}v_1\cdots v_N$.
\end{enumerate}

Then there exist integers $\delta, j$ with $q<j\leq n$ and $j\in \supp(u_{\delta})$ such that $x_q u_{\delta}/x_j \in I$.
\end{Definition}

\begin{Theorem}\label{thmx}
$I_{c,(n,d,t)}$ satisfies the $l$-exchange property with respect to the sorting order $<_{\sort}$.
\end{Theorem}

\begin{proof}
We use similar arguments as in \cite[Proposition 2.2]{bahareh}. Let $t_{u_1}\cdots t_{u_N}, t_{v_1}\cdots t_{v_N}\in K[\{t_u \:\; u\in G(I)\}]$ be two standard monomials of degree $N$ satisfying (i) and (ii) of Definition \ref{lexchange}. Since $t_{u_1}\cdots t_{u_N}, t_{v_1}\cdots t_{v_N}$ does not belong to the initial ideal of $J$ with respect to the sorting order, the products $u_1\cdots u_N$ and $v_1\cdots v_N$ are sorted. Condition (i) together with (\ref{sequence}) implies that
\begin{eqnarray}\label{Condition1}
\text{$\deg_{x_i}(u_{\gamma})=\deg_{x_i}(v_{\gamma})$, for all $1\leq \gamma\leq N$ and $1\leq i \leq q-1$,}
\end{eqnarray}
and Condition (ii) implies that there exists $1\leq \delta\leq N$ such that
\begin{eqnarray}\label{Condition2}
\text{$\deg_{x_q}(u_{\delta})<\deg_{x_q}(v_{\delta})$.}
\end{eqnarray}

Let $u_{\delta}=x_{j_1}\cdots x_{j_d}$,  $v_{\delta}=x_{l_1}\cdots x_{l_d}$. Then from (\ref{Condition1}) and (\ref{Condition2}), we see that there exists $k$ such that $j_1=l_1, \dots, j_{k-1}=l_{k-1}$ and $j_k>l_k$. Then $l_k=q$.   We need to show that $x_qu_{\delta}/x_j \in I_{c,(n,d,t)}$, for some $j\in \supp(u_{\delta})$ with $q<j$.

Take $j=j_k$. Then  $x_qu_{\delta}/x_j= x_{j_1}\cdots x_{j_{k-1}}x_q x_{j_{k+1}}\cdots x_{j_d}=x_{l_1}\cdots x_{l_{k-1}}x_{q}x_{j_{k+1}}\cdots x_{j_d}$.  To see $w=x_qu_{\delta}/x_j$ is $t$-spread, we only need to check $j_{k+1}-q \geq t$. Indeed, $j_{k+1}-q =j_{k+1}-l_k > j_{k+1}-j_k \geq t$. Moreover, $w$ is $c$-bounded because $\{l_1\leq l_2 \leq \ldots \leq q\}\subset \msupp(v_{\delta})$ and $\{j_{k+1} \leq \ldots \leq j_d\} \subset \msupp(u_{\delta})$ are  $c$-bounded, and $j_{k+1}- q >t$.
 \end{proof}

For the sortable ideal $I_{c,(n,d,t)}$  we consider the sorting order $<_{\sort}$ on $T$ and the lexicographic order $<_{\lex}$ on $S$. Let $<$ be the monomial order on $R$ defined as follows: if $s_1, s_2 \in S$ and $t_1, t_2 \in T$ be monomials, then $s_1t_1>s_2t_2$ if $s_1>_{lex}s_2$ or $s_1=s_2$ and $t_1>_{sort}t_2$.

\medskip
Let $P\subset R$ be the defining ideal of $\mathcal{R}(I_{c,(n,d,t)})$. The following result shows that $P$   has a quadratic Gr\"obner basis and the initial ideal of $P$ is squarefree.

\begin{Theorem} \label{GB}
The reduced Gr\"obner basis of the toric ideal $P$ with respect to $<$ defined before  consists of the set of binomials $t_ut_v- t_{u^{\prime}}t_{v^{\prime}}$, where $(u, v)$ is unsorted and $(u^{\prime}, v^{\prime})=\sort(u,v)$ and the set of binomials of the form $x_i t_u-x_j t_v$, where $i<j$, $x_iu=x_jv$, and $j$ is the largest integer for which $x_i v/x_j \in G(I)$.

\begin{proof}
The result follows from Theorem \ref{thmx} and \cite[Theorem 5.1]{HHV}.
\end{proof}
\end{Theorem}

We have the following consequences:
\begin{Corollary} \label{ReesCM}
The Rees algebra $\mathcal{R}(I_{c,(n,d,t)})$ is a normal Cohen-Macaulay domain.
\end{Corollary}

\begin{Corollary}
\label{strongpersistence}
$I_{c,(n,d,t)}$ satisfies the strong persistence property and all powers of $I_{c,(n,d,t)}$ have linear resolution.
\end{Corollary}

\begin{proof}
The first part follows from Theorem \ref{GB} and \cite[Corollary 1.6]{JQ} and the second part follows from Theorem \ref{GB} and \cite[Theorem 10.1.9]{monomialideals}.
\end{proof}

\section{The analytic spread of $c$-bounded $t$-spread Veronese ideals}

By Corollary~\ref{ReesCM}, the Rees algebra of $I_{c,(n,d,t)}$ is Cohen-Macaulay. Therefore  $\lim_{s\to\infty}\depth S/I^s=n-\ell(I_{c,(n,d,t)})$, see \cite[Proposition 3.3]{EisenbudHuneke}.  Here for a graded  ideal $I\subset S$, we denote by  $\ell(I)$  the analytic spread of $I$, which by  definition is the Krull dimension of  $\mathcal{R}(I)/\mm \mathcal{R}(I)$, where $\mm$ denotes the graded maximal ideal of $S$.  The aim of this section is to give an explicit formula for $\ell(I_{c,(n,d,t)})$.

\begin{Definition}
Let $G(I)=\{u_1, \dots, u_m\}$. {\em The linear relation graph} $\Gamma$ of $I$ is the graph with the edge set
\[
E(\Gamma)=\{\{i,j\} \:\; \text{ there exist } u_k, u_l \in G(I) \text{ such that } x_i u_k=x_j u_l\}.
\]
\end{Definition}
 \begin{Definition}
{\em The analytic spread} of an ideal $I$, $l(I)$, is the Krull dimension of the fiber ring $\mathcal{R}(I)/\mathfrak{m}\mathcal{R}(I)$.
\end{Definition}
To compute the analytic spread of $I_{c, (n,d,t)}$, which is the Krull dimension of $K[I_{c, (n,d,t)}]$, we use the following result, which is a generalization of \cite[Lemma~4.2]{JQ}.

\begin{Lemma}
\label{better}
Let $I$ be a monomial ideal with linear relations generated in a single degree whose linear relation graph $\Gamma$ has $r$ vertices and $s$ connected components. Then
\[
\ell(I)=r-s+1.
\]
\end{Lemma}

\begin{proof}
Let $G(I)=\{u_1, \dots, u_q\}$ with $u_i=\mathbf{x}^{a_i}$ for all $i$, and $\mathcal{M}=\{a_1, \dots, a_q\}$ be the set of exponent vectors of $G(I)$. We denote by $V$ the $\mathbb{Q}$-vector space generated by the vectors $a_1, \dots, a_q$. It is clear that $\ell(I)$ is the dimension of $V$.
Let $W \subset V$ be the $\mathbb{Q}$-vector space spanned by all the vectors $a_k-a_l$ with $a_k, a_l \in \mathcal{M}$ such that $a_k-a_l= \pm \varepsilon _{ij}=\pm (\varepsilon_i -\varepsilon_j)$, for some $i<j$. In \cite[Lemma 4.2]{JQ}, it was showed that $\dim W=r-s$.

We show that $a_i-a_j \in W$ for all $i$ and $j$. We may assume $i\neq j$. Since $I$ is multigraded, there exists an exact sequence of multigraded modules

\[
0\To U\To \Dirsum_{i=1}^m Se_i \To I \To 0
\]
with $e_i \mapsto u_i$ for $i=1,\dots,m$ and the multidegree of $e_i$, denoted $\Deg(e_i)$, is equal to $a_i$.

Since $I$ has linear relations, $U$ is generated by homogeneous relations of the form $r=x_ke_{\alpha}-x_le_{\beta}$. Here $\Deg(r)=\varepsilon_k+a_{\alpha}=\varepsilon_l+a_{\beta}$. Hence $a_{\alpha}-a_{\beta}\in W$ for the generating relations of $I$. Note that

\[
r_{ij}=\frac{u_j}{\gcd(u_i, u_j)}e_i-\frac{u_i}{\gcd(u_i, u_j)}e_j \in U.
\]
Thus, we can write
\[
r_{ij}=\sum_{t=1}^m \widetilde{w}_{t}(x_{k_{t}}e_{\alpha_{t}}-x_{l_{t}}e_{\beta_{t}}),
\]
with $\Deg(\widetilde{w}_{t}(x_{k_{t}}e_{\alpha_{t}}-x_{l_{t}}e_{\beta_{t}}))= \Deg(r_{ij})$ for all $t$.
Thus
\[
r_{ij}=\sum_{t=1}^m (v_{t}e_{\alpha_{t}}-w_{t}e_{\beta_{t}}),
\]
where $v_t= \widetilde{w}_tx_{k_t}$ and $w_t=\widetilde{w}_tx_{l_t}$.

 Moreover, all summands are homogeneous of multidegree $\Deg(r_{ij})$.

Since $r_{ij}$ contains only the basis elements $e_i$ and $e_j$, the other basis elements in this sum must cancel each other. Therefore, we can write this sum as
\[
r_{ij}=\sum_{t=1}^m (v_{t}e_{\alpha_{t}}-w_{t}e_{\alpha_{t+1}}),
\]
where $\alpha_1=i$ and $\alpha_m=j$. Since $w_{t}e_{\alpha_{t+1}}$ has the same multidegree as $v_{t+1}e_{\alpha_{t+1}}$, it follows that $w_t=v_{t+1}$, for all $i=1,\dots, m-1$. Thus
\[
r_{ij}=\sum_{t=1}^m (v_{t}e_{\alpha_{t}}-v_{t+1}e_{\alpha_{t+1}}).
\]
Let $b_t=\Deg v_t$ for all $t$. Then

\[
b_{t+1}-b_t= \Deg(e_{\alpha_t})-\Deg(e_{\alpha_{t+1}}) \in U.
\]
Therefore,

\[
a_i-a_j=b_m-b_1= \sum_{t=1}^m (b_{t+1}-b_t) \in W.
\]

Since $a_1, \dots, a_q$ is a $\mathbb{Q}$-basis of $V$, it follows that $a_1\notin W$. Hence, $V=W+a_1\mathbb{Q}$, which leads to

\[
\ell(I)=\dim V= \dim W+1=r-s+1.
\]
\end{proof}

Now we apply this Lemma \ref{better} to obtain an explicit formula for the analytic spread of $I_{c,(n,d,t)}$. In order to do this, we first introduce a partial order $\prec$ on $G(I_{c,(n,d,t)})$. Let $u, v\in G(I_{c,(n,d,t)})$. We say {\it $u$ covers $v$} with respect to $\prec$ if there exists $i<j$ such that $x_j$ divides $u$ and $v=x_i\frac{u}{x_j}$.

\begin{Theorem}\label{maxelem}
For integers $d>0$ and $c>0$, write  $d=kc+r$ with integers $k\geq 0$ and  $0<r\leq c$. In other words, $k=\lfloor (d-1)/c\rfloor$. Then the following holds:

The partially ordered set $G(I_{c,(n,d,t)})$ has a unique maximal element, namely $x_{a_1}\cdots x_{a_d}$ with
 \[
a_{i}=n-k-t(d-1)+(i-1)t+\left\lceil \dfrac{i-r}{c}\right \rceil =a_1 +(i-1)t+\left\lceil \dfrac{i-r}{c}\right  \rceil
\]
for $i=1,\ldots,d$.  In other words,
\[
a_{d-i+1}=n-(i-1)t-\left \lfloor \frac{i-1}{c}\right \rfloor,
\]
 and it has  a unique minimal element, namely $x_{{\alpha}_1}\cdots x_{{\alpha}_d}$ with
\[
{\alpha}_i=(i-1)t+ \left\lfloor \frac{i-1}{c}\right\rfloor +1,
\]
for any $i=1,\dots,d$.
\end{Theorem}

\begin{proof}
Denote by $u_0=x_{a_1}\cdots x_{a_d}$. It is clear that $u_0 \in G(I_{c,(n,d,t)})$. Suppose that $u_0$ is not a maximal element. Then there exists an element $x_{b_1}\dots x_{b_d}$ that covers $u_0$. Thus, there exist $i<j$ such that $x_j \mid x_{b_1}\cdots x_{b_d}$ and $u_0=x_i(x_{b_1}\cdots x_{b_d})/x_j$, which implies that $x_j u_0/ x_i= x_{b_1}\cdots x_{b_d}$, which is not possible and this can be seen from the formula for $a_{d-i+1}$.

In order to show the uniqueness of $u_0$, we assume that there exists another element $u=x_{c_1}\cdots x_{c_d}$ which is maximal and $u\neq u_0$. Since $u\neq u_0$, there  exists some $l$ such that $c_l<a_l$ and we take the largest $l$ with this property. Then $u^{\prime}=x_{a_l}u/x_{c_l} \in I_{c,(n,d,t)}$ and $u\prec u^{\prime}$, hence a contradiction.

Similarly one can argue that $x_{{\alpha}_1}\cdots x_{{\alpha}_d}$ is the unique minimal element with respect to $\prec$.
\end{proof}

By $[a,b]$ we denote all the integer numbers $c$ with $a\leq c \leq b$. With the notation from Theorem \ref{maxelem}, we have the following result:
\begin{Proposition}
\label{complete}
For $i=1,\dots, d$, let $K_i$ be the complete graph on $[\alpha_i, a_i]$. Let $\Gamma$ be the linear relation graph of $I_{c,(n,d,t)}$. Without loss of generality we may assume that $\gcd(G(I_{c,(n,d,t)}))=1$.  Then $\alpha_i<a_i$ for all $i$,  and  $\Gamma$ and $\Union_{i=1}^d K_i$ have the same number  of connected components and the same number of vertices.
\end{Proposition}

\begin{proof} Suppose $\alpha_i=a_i$ for some $i$. Then $x_{a_i}$ is common factor of the elements in $G(I_{c,(n,d,t)})$, a contradiction.

We show that $\Union_{i=1}^d K_i\subseteq \Gamma$.  Indeed,
let $\{k,l\}\in \Union_{i=1}^d K_i$. Then $\{k,l\}\in K_i$ for some $i$. We may assume $k< l$. Thus, ${\alpha}_i \leq k< l\leq a_i$. Let $w_1=x_{{\alpha}_1}\cdots x_{{\alpha}_{i-1}}x_k x_{a_{i+1}}\cdots x_{a_d}$ and $w_2=x_{{\alpha}_1}\cdots x_{{\alpha}_{i-1}}x_l x_{a_{i+1}}\cdots x_{a_d}$. Then, it is clear that $w_1$ and $w_2 \in G(I_{c,(n,d,t)})$ and $x_lw_1=x_kw_2$ is a relation, thus $\{k,l\}\in E(\Gamma)$.

Next we observe that $V(\Gamma)=V(\Union_{i=1}^d K_i)$. We only need to show that $V(\Gamma)\subseteq V(\Union_{i=1}^d K_i)$. In fact, if $i\in V(\Gamma)$, then there exists $u\in G(I_{c,(n,d,t)})$ with $i \in \supp(u)$. Note that if  $u=x_{i_1}\cdots x_{i_d}$, then $i_q\in K_q$ for all $q$, by Theorem~\ref{maxelem}. This show that $i \in K_q$ for some $q$.

Now since $\Union_{i=1}^d K_i\subset \Gamma$ and since $V(\Gamma)=V(\Union_{i=1}^d K_i)$ it follows that the number of connected components of $\Gamma$ is less than or equal to the number of connected components $\Union_{i=1}^d K_i$.

Now let $\{i,j\}\in E(\Gamma)$. We show that the vertices $i$ and $j$ belong to the same connected component of $\Union_{i=1}^d K_i$. This then shows that $\Gamma$ and $\Union_{i=1}^d K_i$ have the same number of connected components. Indeed, we may assume that $i<j$. Because $\{i,j\}\in E(\Gamma)$, there exists a monomial $u\in G(I_{c, (n,d,t)})$ such $x_j| u$ and $v=x_i(u/x_j)\in G(I_{c, (n,d,t)})$. Let $u=x_{i_1}\cdots x_{i_k}\cdots x_{i_d}$. Then  $j=i_k$ for some $k$. By Theorem~\ref{maxelem},  $i_q \in K_q=[\alpha_q, a_q]$ for all $q=1,\dots, d$. Let $l$ be the smallest integer such that $i\leq i_l$. If $k=l$, then $i,j\in K_l$, and hence to the same connected component. Now let $k\neq l$.
Then
\[
v=x_{i_1}\cdots x_{i_{l-1}}x_i x_{i_l}\cdots x_{i_{k-1}}x_{i_{k+1}}\cdots x_{i_d}.
\]

Since $v\in  G(I_{c,(n,d,t)})$, it follows that $i_q\in K_{q+1}$ for $q=l,\ldots k-1$. On the other hand $i_q\in K_q$ for all $q$. It follows that $K_q\sect K_{q+1}\neq \emptyset$ for $q=l,\ldots, k-1$. Therefore,
\[
K_l\union K_{l+1}\union \cdots \union  K_k
\]
is connected. Since $i\in K_l$ and $j\in K_k$, it follows that $i$ and $j$ belong to the same connected component.
\end{proof}

\begin{Corollary}
\label{connected}
We have  $\ell(I_{c,(n,d,t)})=n$ if and only if $\Gamma$ is connected.
\end{Corollary}

\begin{proof}
By Lemma~\ref{better} and Proposition~\ref{complete} we have  $\ell(I_{c,(n,d,t)})=|V(\Gamma)|-s+1$, where $s$ is the number of connected components of $\Gamma$.

Hence if  $\ell(I_{c,(n,d,t)})=n$,  we must have $|V(\Gamma)|=n$ and $s=1$, because $|V(\Gamma)|\leq n$ and $s\geq 1$,  in general. Hence $\Gamma$ is connected.

Conversely, assume that $\Gamma$ is connected. Then $V(\Gamma)=\Union_{i=1}^dK_i=[\alpha_1,a_d]$.
Since $\alpha_1=1$ and $a_d=n$, we see that $|V(\Gamma)|=n$, and hence $\ell(I_{c,(n,d,t)})=n$.
\end{proof}

For integers $d>0$ and $c>0$,  let $r=d-\lfloor (d-1)/c\rfloor c$. Then for $i=1,\dots, d-1$ we set
\[
\delta_i= \left\lfloor \frac{i}{c}\right\rfloor - \left\lceil\frac{i-r}{c}\right\rceil.
\]

For the proof of the next results  we need
\begin{Lemma} \label{delta_i}
Let $d=kc+r$ with $0<r\leq c$ and $i-r=lc+s_i$ with $0\leq s_i<c$ for some integers $k$ and $l$. Then we have:
\begin{enumerate}[{\em(i)}]
\item if $s_i=0$, then
\[
\delta_i= \left\{
\begin{array}{ll}
      0, & \text{ if } r<c,\\
      1, & \text{ if } r=c, \\
\end{array}
\right.
\]
\item if $0<s_i<c$ and  $r=c$,  then $\delta_i=0$, while  if $r<c$,  then
\[
\delta_i= \left\{
\begin{array}{rl}
      -1, & \text{ if } r+s_i<c,\\
      0, & \text{ if } r+s_i\geq c. \\
\end{array}
\right.
\]

\end{enumerate}
\end{Lemma}

\begin{proof}
(i) If $s_i=0$, then $c\mid (i-r)$. Then
\[
\delta_i =\left\lfloor\frac{i}{c}\right\rfloor-\left\lfloor\frac{i-r}{c}\right\rfloor.
\]
 Hence, if $r=c$, then $i=(l+1)c$, which implies that $\delta_i=l+1-l=1$. If $r<c$, then $i=lc+r$ with $0<r<c$, which implies that $\delta_i=l-l=0$.

(ii) If $0<s_i<c$, then $c\nmid (i-r)$. Then
\[
\delta_i= \left\lfloor \frac{i}{c}\right\rfloor - \left\lfloor\frac{i-r}{c}\right\rfloor-1.
\]
Now, if $r=c$, then $i=(l+1)c+s_i$ with $0< s_i<c$, thus $\delta_i=l+1-l-1=0$. If $r<c$, then $i=lc+r+s_i$ with $0<r<c$ and $0< s_i<c$. Thus $0< r+s_i< 2c$ and we can distinguish the cases: if $0<r+s_i<c$, then $\delta_i=-1$, and if $c\leq r+s_i <2c$, then $\delta_i=0$.
\end{proof}

The proof of the previous lemma shows

\begin{Lemma}\label{delta}
Let $\delta_{\min}= \min\{ \delta_i\:\;  i=1, \dots, d-1\}$ and $\delta_{\max}= \max\{ \delta_i\:\;  i=1, \dots, d-1\}$.   Then

\[
\delta_{\min}= \left\{
\begin{array}{rl}
      -1, & \text{ if } r<c,\\
        0, & \text{ if } r=c, \\
\end{array}
\right.
\]

\[
\delta_{\max}= \left\{
\begin{array}{ll}
      0, & \text{ if } r<c,\\
      1, & \text{ if } r=c, \\
\end{array}
\right.
\]

\end{Lemma}

\begin{Proposition}\label{dcomp}
Let $k=\left\lfloor (d-1)/c\right\rfloor$. Then the linear relation graph of $I_{c,(n,d,t)}$ has at most $d$ connected components,  and it has exactly $d$ connected components if and only if $n-(k+t(d-1))<t+1+\delta_{\min}$. In this case, the connected components of $\Gamma$ are $K_1,\dots, K_d$.
\end{Proposition}

\begin{proof}
By Proposition~\ref{complete}, it follows that $\Gamma$ has at most $d$ connected components and that it has exactly $d$ connected components if and only if $V(K_i)\cap V(K_j)= \emptyset$ for all $i\neq j$. We may assume that $i<j$. Then $V(K_i)\cap V(K_j)= \emptyset$ for $i<j$ if and only if $V(K_i)\cap V(K_{i+1})=\emptyset$ for any $i$, because ${\alpha}_j>a_{i+1}$. Now $V(K_i)\cap V(K_{i+1})=\emptyset$ if and only if $a_i<{\alpha}_{i+1}$. Therefore, together with  Proposition~\ref{complete} it follows that $\Gamma$ has exactly $d$ components if and only if $a_i<{\alpha}_{i+1}$ for all $i$,  and in this  case  $\Gamma =\Union_{i=1}^d K_i$ and  $K_1,\ldots, K_d$ are the connected components of $\Gamma$.

Now we analyze what it means that $a_i<{\alpha}_{i+1}$ for all $i$. Let $r=d-kc$. By Theorem~\ref{maxelem}, the condition $a_i<\alpha_{i+1}$ holds if and only if $$a_1<t+1+ \left\lfloor \frac{i}{c}\right\rfloor -\left\lceil \frac{i-r}{c}\right\rceil$$ for all $i$, and this is the case if and if  $a_1<t+1+\delta_{\min}.$
\end{proof}

\begin{Theorem}\label{hardwork}
Let $\Gamma$ be the relation graph of $I_{c,(n,d,t)}$, $k= \left\lfloor (d-1)/c\right\rfloor$ and $r=d-kc$. Then the following holds:
\begin{enumerate}[{\em(i)}]
\item If $n-k-t(d-1)<t+1+\delta_{\min}$, then $\Gamma$ has exactly $d$ connected components and
\[
|V(\Gamma)|= nd-d(d-1)t-k(k-1)c-2rk.
\]
\item $\Gamma$ is a connected graph and has $n$ vertices if and only if  $n-k-t(d-1)\geq t+1+\delta_{\max}$.

\item If  $t+1+\delta_{\max}>a_1 \geq t+1+\delta_{\min}$,
then $\Gamma$ has $d/c$ connected components and $n$  vertices  if $r=c$, and  $\Gamma$ has $(r+1)k+1$ connected components and
\[
\nu=n-\sum_{j=0}^{k}\sum_{i=0}^{r+1}(\alpha_{jc+i}-a_{jc+i-1}+1)-(\alpha_{kc+r+1}-a_{kc+r}+1)
\]
vertices  if $r<c$.
\end{enumerate}
\end{Theorem}

\begin{proof}
(i) By Theorem~\ref{maxelem}, $a_1=n-k-t(d-1)$, and since $a_1<t+1+\delta_{\min}$, Proposition~\ref{dcomp} implies that  $\Gamma$ has exactly $d$ connected components.
By Proposition~\ref{complete}, and by using Theorem~\ref{maxelem} we obtain
\begin{eqnarray*}
|V(\Gamma)|&=& \sum_{i=1}^d |[{\alpha}_i, a_i]|=\sum_{i=1}^d(a_i-{\alpha}_i+1)\\
&=& \sum_{i=1}^d a_{d-i+1} -\sum_{i=1}^d(\alpha_i-1)=\sum_{i=1}^{d}(a_{d-i+1}-\alpha_i +1)\\
&=& \sum_{i=1}^d (n-2t(i-1)-2\left\lfloor \frac{i-1}{c}\right\rfloor)\\
&=& nd-2t{d\choose 2}-2\sum_{i=1}^d \left\lfloor \frac{i-1}{c}\right\rfloor\\
&=& nd -2t{d\choose 2}-2({k\choose 2}c+rk)\\
&=& nd-d(d-1)t-k(k-1)c-2rk.
\end{eqnarray*}

(ii) By Proposition~\ref{complete}, it follows that $\Gamma$ is connected if and only if $V(K_i)\cap V(K_{i+1})\neq \emptyset$ for all $i=1,\dots, d-1$ which is the case if and only if  ${\alpha}_{i+1}\leq a_i$ for all $i$.
  We have that ${\alpha}_{i+1}\leq a_i$ if and only if $a_1\geq t+1+\delta_{\max}$, see Theorem~\ref{maxelem}.

(iii)  By Lemma~\ref{delta},  $\delta_{\max} - \delta_{\min}=1$. Therefore,  our assumption implies that  $a_1=t+1+\delta_{\min}$. It follows that
\[
a_i-\alpha_{i+1}= a_1-t-1-\delta_{i}
\leq 0.
\]
Thus, we see that $V(K_i)\sect V(K_{i+1})\neq \emptyset$ if and only if  $a_i=\alpha_{i+1}$. This is the case if only $\delta_{\min}=\delta_i$.

Assume $r=c$. By Lemma~\ref{delta}, $\delta_{\min}=0$. We have to consider those $i \in \{1, \dots, d-1\}$ such that $\delta_i=\delta_{\min}=0$. Hence, by using Lemma~\ref{delta_i}, we obtain
\[
\{i\:\; \delta_i = \delta_{\min}=0\}=\{i\:\; c\nmid (i-r)\}=\{i\:\; d-i \not\equiv 0 \mod c\}=\{i\:\; i \not\equiv d \mod c\}.
\]
Since $d \equiv 0\mod c$, it follows that $V(K_i)\sect V(K_{i+1})\neq \emptyset$  if and only if $i\not\equiv 0\mod c$. This shows that in this case we have $ d/c$ components and $|V(\Gamma)|=n$.

Finally assume that $r<c$. By Lemma~\ref{delta}, $\delta_{\min}=-1$. We have to consider those $i \in \{1, \dots, d-1\}$ such that $\delta_i=\delta_{\min}$. Hence, by using Lemma~\ref{delta_i}, we obtain
\[
\{i\:\; \delta_i=\delta_{\min}=-1\}.
\]
It follows form the definition of $\delta_i$ that $\delta_i=-1$ if and only if $i\in [jc+r+1, (j+1)c-1]$ for $j=0,\ldots, k-1$. From this we deduce that $\Gamma$ has $(r+1)k+1$ connected components and
\[
n-\sum_{j=0}^{k}\sum_{i=0}^{r+1}(\alpha_{jc+i}-a_{jc+i-1}+1)-(\alpha_{kc+r+1}-a_{kc+r}+1)
\]
 vertices,  where $a_{-1}= 0$, by definition.
\end{proof}

\begin{Theorem}\label{headache}
Let $k= \left\lfloor (d-1)/c\right\rfloor$ and $r=d-kc$. Then the following holds:

\begin{enumerate}[{\em(i)}]
\item If $n-k-t(d-1)<t+1+\delta_{\min}$, then
\[
\ell(I_{c,(n,d,t)})= nd-d(d-1)t-k(k-1)c-2rk-d+1.
\]

\item If $n-k-t(d-1)\geq t+1+\delta_{\max}$,  then $\ell(I_{c,(n,d,t)})=n$.

\item If $t+1+\delta_{\max}>a_1 \geq t+1+\delta_{\min}$, then $\ell(I_{c,(n,d,t)})=n-d/c+1$, if $r=c$, and $\ell(I_{c,(n,d,t)})=\nu-(r+1)k$, if $r<c$.
\end{enumerate}
\end{Theorem}

\begin{proof}
The desired result  follows from Lemma~\ref{better} and Theorem~\ref{hardwork}.
\end{proof}

We illustrate Theorem~\ref{headache} with suitable examples.

\begin{Examples}
(i) Let $I=I_{3, (12, 4,3)}$. Then $\Gamma$ has $4$ connected components: $K_1$ is the complete graph  on $[1,2]$, $K_2$ on $[4,6]$, $K_3$ on $[7,9]$ and $K_4$ on $[11,12]$. Since $\Gamma$ has $10$ vertices, we obtain $\ell(I)=10-4+1=7$.

(ii) Let $I=I_{3,(16,6,2)}$. Then $\Gamma$ is connected and has $16$ vertices. Thus, $\ell(I)=16-1+1=16$.

(iii) Let $I=I_{2,(5,3,1)}$. Then $\Gamma$ is connected and has $5$ vertices. Thus, $\ell(I)=5-1+1=5$.

(iv) Let $I=I_{2,(9,6,1)}$. Then $\Gamma$ has $3$ connected components: the connected component $K_1\union K_2$ with vertex set  $[1,3]$,   $K_3\union K_4$ with vertex set  $[4,6]$ and $K_5\union K_6$ on the vertex set  $[7,9]$. Since $\Gamma$ has $9$ vertices, we obtain $\ell(I)=9-3+1=7$.
\end{Examples}

Let $S$ be a standard graded $K$-algebra of dimension $n$ and $I\subset S$ a graded ideal. By a theorem of Brodmann \cite{brodmann}, $\depth S/I^s$ is constant for all $s$ large enough. This constant value is called the {\it limit depth} of $I$. Brodmann showed that
\[
\lim_{s\to\infty}\depth S/I^s\leq n-\ell(I).
\]
If the Rees algebra of $I$ is Cohen-Macaulay, then by a result of Huneke \cite{Huneke}, the associated graded ring of $I$ is Cohen-Macaulay and, by using a result of Eisenbud and Huneke \cite[Proposition 3.3]{EisenbudHuneke}, this implies that
\[
\lim_{s\to\infty}\depth S/I^s= n-\ell(I).
\]
We obtain the following consequence:

\begin{Corollary}
 $\lim_{s\to\infty}\depth S/I_{c,(n,d,t)}^s=0$ if and only if condition (ii) of  Theorem~\ref{hardwork} holds, or $c=d$ and condition (iii) of   Theorem~\ref{hardwork} holds.
\end{Corollary}

\begin{proof}
By Corollary~\ref{ReesCM}, the Rees algebra of $I_{c,(n,d,t)}$ is Cohen-Macaulay. Therefore  $\lim_{s\to\infty}\depth S/I^s=0$ if and only if $\ell(I_{c,(n,d,t)})=n$. By Corollary~\ref{connected},  this is the case if and only if $\Gamma$ is connected.  Thus  the desired conclusion follows from Theorem~\ref{hardwork}.
\end{proof}

\section{$t$-spread Veronese ideals of bounded block type}\label{tVeroblock}

In this section we study properties of $t$-spread Veronese ideals of bounded block type. Recall that $I_{(n,d,t),k}$ is the $t$-spread  monomial ideal in $n$ variables generated in degree $d$ with at most $k$ blocks. In particular,  $I_{(n,d,0),k}$ is the ideal in $n$ variables generated in degree $d$  whose generators have  support with  cardinality at most  $k$.

We have $k\leq n$, and if $k=n$, then $I_{(n,d,t),k}=I_{n,d,t}$. Moreover, $I_{(n,d,0),1}=(x_1^d,\ldots,x_n^d)$. In the following we assume that $1<k<n$ and $d\geq 2$.

The next result provides a description of the ideals $I_{(n,d,0),k}$ and has the consequence that for arbitrary $t$ the ideal  $(I_{(n,d,t),k})$ can be obtained by an iterated application of the Kalai shifting operator.
\begin{Lemma}\label{sig}
We have
\begin{enumerate} [{\em(i)}]
\item $I_{(n,d,0),k} = \sum _{i_1 \leq i_2 \leq \ldots \leq i_k} (x_{i_1}, \ldots, x_{i_k})^d$, and
\item $(I_{(n,d,t),k})^{\sigma}=I_{(n+(d-1),d,t+1),k}$.
\end{enumerate}
\end{Lemma}

\begin{proof}
(i) obvious by definition.

(ii) Let $u\in I_{(n,d,t),k}$, $u=x_{A}, A=\{i_1\leq \dots \leq i_d\}$. Since $A$ is $t$-spread, it follows that $A^{\sigma}$ is $t+1$-spread. Moreover, if $B_1\sqcup \dots \sqcup B_r$, is the block decomposition of $A$, then the block decomposition $A^{\sigma}$ is  $C_1\sqcup\cdots \sqcup C_r$ such that if $B_j=\{i_k \leq i_{k+1} \leq \ldots \leq {i_l}\}$ then $C_j= \{i_k+k-1 < i_{k+1}+k < \ldots < i_l+l-1\}$.  Thus, $u^{\sigma}\in I_{(n+(d-1), d, t+1),k}$. Similarly one shows that if $ v\in I_{(n+(d-1), d, t+1),k} $, then $u=v^{\tau}\in I_{(n, d, t),k}$. Since $u^{\sigma}=v$, this completes the proof.

\end{proof}

The following result is a consequence of Lemma~\ref{sig}(ii).
\begin{Corollary}
$I_{(n,d,t),k}=(I_{(n-(d-1)t,d,0),k})^{\sigma^t}$.
\end{Corollary}

\begin{Proposition}
The number of generators of $I_{(n,d,0),k}$ is
\[
\mu(I_{(n,d,0),k})= \sum_{i=1}^{k} {d-1 \choose i-1}{n\choose i}.
\]
\end{Proposition}
\begin{proof}
The number of monomials in $n$ variables of degree $d$ whose support have cardinality $i$ is ${d-1 \choose i-1}{n\choose i}$. Thus the desired formula follows.
\end{proof}

\begin{Theorem}
\label{regblock}
\[
\reg (I_{(n,d,0),k})= (n-k)\left\lfloor \frac{d-1}{k} \right\rfloor +d-1.
\]
\end{Theorem}
\begin{proof}

Note that $\dim S/I_{(n,d,0),k}=0$, since $x_i^d\in I_{(n,d,0),k}$ for $i=1,\ldots,n$. Therefore, the largest degree of a monomial in $S\setminus I_{(n,d,0),k}$ gives us the regularity of $I_{(n,d,0),k}$.

Let $I=\{i_1, \dots, i_l\}\subset [n]$, $l\leq k$. Denote by $a_{I}=a_{i_1}+\dots+a_{i_l}$.
Let us consider the polytope

\[
\mathcal{P}= \{(a_1, \dots, a_n)\in \mathbb{R}^{n}: a_i\geq 0, a_{I}\leq d-1 \text{ for all } I\subset [n], |I|\leq k\}.
\]
Note that $x_1^{a_1}\cdots x_n^{a_n}\not \in I_{(n,d,0),k}$ if and only if $(a_1,\ldots,a_n)\in \mathcal{P}$. Thus,
in order to find the maximal degree of the socle elements, we have to find an integer point $a=(a_1, \dots, a_n) \in \mathcal{P}$ for which the function $f(a_1, \dots, a_n)=a_1+\dots+a_n$ takes the maximal value. By symmetry, we can assume that $a_1\geq a_2 \geq \dots \geq a_n$. Then $a\in \mathcal{P}$ if and only if $a_1+a_2+\dots+a_k\leq d-1$ and $a_i \geq 0$ for all $i$. In particular, if $a\in \mathcal{P}$, then $a^{\prime}=(a_1,\dots, a_{k-1}, a_k, \dots, a_k)\in \mathcal{P}$ and $f(a^{\prime})\geq f(a)$. Therefore, in order to find the maximal value for $f$, we may assume that $a=(a_1,\dots, a_{k-1}, a_k,\dots, a_k)$. For such $a$, we have that $f(a)=a_1+\dots+a_{k-1}+ (n-k+1)a_k$. Now, our problem can be reformulated as follows: we are looking for $a_1, \dots, a_k$ with the following conditions: find the maximal value of the linear function $$g(a_1,\dots, a_k)=a_1+\dots+a_{k-1}+(n-k+1)a_k$$
for the  integer points $(a_1,\dots, a_k)$ satisfying
\begin{enumerate}[{(i)}]
\item $a_1\geq a_2\geq \dots \geq a_k>0$,
\item $a_1+a_2+\dots +a_k\leq d-1$.
\end{enumerate}
Now let $a_k=i$ with $1\leq i\leq d-1$. By (i) and (ii), we obtain that $ki\leq d-1$. This implies that $1\leq i\leq \left\lfloor \frac{d-1}{k}\right \rfloor$. Let $h_i= (n-k)i+(d-1)\geq g(a_1, \dots, a_{k-1}, i)$. So for each $i$, $g$ can take at most the value $h_i$. Note that $h_1<h_2<\dots < h_{\left\lfloor \frac{d-1}{k}\right \rfloor}$. We will show that $g(a_1,\dots, a_{k-1},\left\lfloor \frac{d-1}{k} \right\rfloor)$ attains the value $$h_{\left\lfloor \frac{d-1}{k}\right \rfloor}= (n-k)\left\lfloor \frac{d-1}{k}\right \rfloor +(d-1).$$ By definition, $g(a_1, \dots, a_{k-1},\left\lfloor \frac{d-1}{k} \right\rfloor)= a_1+\dots +a_{k-1}+\left\lfloor \frac{d-1}{k}\right \rfloor+(n-k)\left\lfloor \frac{d-1}{k} \right\rfloor$. Indeed, let $a_2=\dots =a_{k-1}=\left\lfloor \frac{d-1}{k}\right \rfloor$ and $a_1=(d-1)-(k-1)\left\lfloor \frac{d-1}{k}\right \rfloor$.
\end{proof}

\begin{Proposition} \label{k=2}
$I^{n-1}_{(n,d,0),2}= \mathbf{m}^{(n-1)d}$.
\end{Proposition}

\begin{proof}
The inclusion $I^{n-1}_{(n,d,0),2}\subseteq \mathbf{m}^{(n-1)d}$ is trivial. Conversely,  let $w=x_1^{a_1}\cdots x_n^{a_n} \in \mathbf{m}^{(n-1)d}$. Thus $\sum_{i=1}^n a_i =(n-1)d$. We show that $w$ can be written as a product of $n-1$ monomials of degree $d$ whose  support has cardinality  at most $2$.  We prove this by induction on $n$. We may assume $a_1\geq a_2\geq \dots \geq a_n$. It is clear that $a_n<d$, because otherwise $\sum_{i=1}^n a_i >(n-1)d$.

Now, if $a_1+a_n<d$, then $a_i+a_n<d$ for all $2\leq i\leq n-1$, hence $\sum_{i=1}^{n-1}a_i+(n-1)a_n<(n-1)d$, which is not possible since $n>2$, by assumption. Thus, $a_1+a_n\geq d$, therefore $a_1-(d-a_n)\geq 0$. We can write
\[
w=(x_1^{d-a_n}x_n^{a_n})(x_1^{a_1-(d-a_n)}x_2^{a_2}\cdots x_{n-1}^{a_{n-1}}).
\]
Denote by $w^{\prime}=x_1^{a_1-(d-a_n)}x_2^{a_2}\cdots x_{n-1}^{a_{n-1}}$. It is clear that $w^{\prime}$ has degree $(n-2)d$ and $|\supp (w^{\prime})|\leq n-1$. By induction hypothesis $w'\in I^{n-2}_{(n-1,d,0),2}$. This implies that $w\in I^{n-1}_{(n,d,0),2}$.
\end{proof}

One can ask in general whether a power of a $0$-spread Veronese ideal of bounded block type is a power of the maximal ideal. This is indeed the case,  since we have the following inclusions
\[
(\mathbf{m}^d)^{n-1} \supseteq I^{n-1}_{(n,d,0),k} \supseteq I^{n-1}_{(n,d,0),2}= (\mathbf{m}^d)^{n-1},
\]
where the last equality is given by  Proposition \ref{k=2}. Therefore,  $I^{n-1}_{(n,d,0),k}=(\mathbf{m}^d)^{n-1}$.  Hence the smallest integer $j$  for which $I^{j}_{(n,d,0),k}=(\mathbf{m}^d)^{j}$ is bounded by $n-1$.

To know the smallest integer  $j$  for which
$I^j_{(n,d,0),k}= \mathbf{m}^{jd}$ is of interest by the following result.

\begin{Proposition}
\label{fibercone} Let $j$ the smallest integer such that  $I^m_{(n,d,0),k}= \mathbf{m}^{md}$. Then
\begin{eqnarray*}
\reg(K[I_{(n,d,0), k}])
&=&\max \{j-1, \reg(S^{(d)})\}\\
&=&\max\{j-1, n-\left\lceil \frac{n}{d}\right\rceil \}.
\end{eqnarray*}
\end{Proposition}

\begin{proof}
Let $S^{(d)}$ be the $d$th Veronese subring of $S=K[x_1,\dots, x_n]$. Set $R=K[I_{(n,d,0), k}]$. Then $R\subset S^{(d)}$. Consider the short exact sequence
\[
0 \to R \to S^{(d)} \to S^{(d)}/R \to 0,
\]
and its induced long exact cohomology sequence
\[
\cdots \to H_{\mm}^0(S^{(d)}/R) \to H_\mm^1(R) \to H_\mm^1(S^{(d)}) \to \cdots
\]
We have that $\dim (S^{(d)})=n$ and
\[
H_\mm^i(S^{(d)})=0 \text{ for all } i\neq n.
\]
Thus, we obtain that
\[
H_\mm^i(R)= \left\{
\begin{array}{lll}
      H_\mm^0(S^{(d)}/R), & \text{ if } i=1,\\
      H_\mm^n(S^{(d)}), & \text{ if } i=n,\\
      0, & \text{ if } i\neq 1, n.
\end{array}
\right.
\]

Thus, $\reg (R)=\max\{j-1, \reg(S^{(d)})\}$. Since $S^{(d)}$ is Cohen Macaulay of dimension $n$, it follows that $\reg{S^{(d)}}=a(S^{(d)})+n$, where $a(S^{(d)})$ is the $a$-invariant of $S^{(d)}$. On the other hand,
\[
a(S^{(d)})=-\min\{i \:\; (\omega_{S^{(d)}})_i \neq 0\},
\]
here $\omega_{S^{(d)}}$ denotes the canonical module of $S^{(d)}$.

Since $(\omega_{S^{(d)}})_i= (\omega_{S})_{id}$ ( \cite[Exercise 3.6.21]{bh}) and since $\omega_S= S(-n)$, it follows that
\[
a(S^{(d)})=-\min\{i \:\; S_{id-n} \neq 0\}= -\left\lceil \frac{n}{d}\right\rceil.
\]
This yields the desired conclusion.
\end{proof}

We expect that $R=K[I_{(n,d,0),k}]$ has quadratic relations. Indeed, if $k=1$, then $R$ is a polynomial ring and if $k=n$, then $R$ is the Veronese algebra which is known to have quadratic relations, see for example \cite[Proposition 6.11 and Theorem 6.16]{EH}. So, if $n=3$ only the case $k=2$ is of interest. In this case, if $d=3$, then $R$ is the pinched Veronese which is known to be even Koszul, see \cite{caviglia}. Another case that we could check with the computer is $n=4, d=2, k=2$,  and this case  $R$ is Gorenstein with quadratic relations, and by our formula, $\reg(R)=2$.

\end{document}